\documentclass{amsart}
\usepackage[utf8]{inputenc}
\usepackage{amssymb}
\usepackage{graphicx}
\usepackage{subcaption}
\usepackage{enumerate}
\usepackage[colorinlistoftodos]{todonotes}
\usepackage{verbatim}
\usepackage{amsthm, url}
\usepackage{xcolor}
\usepackage{tikz-cd}
\usepackage{tikz,pgf}
\usetikzlibrary{decorations.text}
\usepackage{amsmath}
\usepackage{mathtools}
\usepackage{mdframed}

\newtheorem{theorem}{Theorem}[section]
\newtheorem*{fibering lemma}{Fibering Lemma}
\newtheorem*{decomposition lemma}{Decomposition Lemma}
\newtheorem*{hurewicz theorem}{Hurewicz Theorem}

\newtheorem{proposition}[theorem]{Proposition}
\newtheorem{corollary}[theorem]{Corollary}

\theoremstyle{definition}

\newcommand{\df}[1]{{{\bf #1}}}

\newcommand{\RR}{\mathbb{R}}

\newcommand{\Ball}{\bar{B}}
\renewcommand{\epsilon}{\varepsilon}

\newcommand{\cliff}[1]{\todo[inline, color=yellow!40]{Cliff todo: #1}}

\newcommand{\VR}{\mathrm{VR}}

\DeclareMathOperator{\diam}{diam}

\DeclareMathOperator{\Cech}{\check{C}ech}
\DeclareMathOperator{\Conv}{Conv}
\DeclareMathOperator{\dgm}{dgm}

\begin{document}

\bibliographystyle{abbrv}

\title[Weighted persistent homology]{Weighted persistent homology}
\author{G.~Bell}
\address{Department of Mathematics and Statistics, The University of North Carolina at Greensboro, Greensboro, NC 27402, USA} 
\email{gcbell@uncg.edu}

\author{A.~Lawson}
\address{Department of Mathematics and Statistics, The University of North Carolina at Greensboro, Greensboro, NC 27402, USA} 
\email{azlawson@uncg.edu}

\author{J.~Martin}
\address{Department of Mathematics and Statistics, The University of North Carolina at Greensboro, Greensboro, NC 27402, USA} 
\email{jmmart27@uncg.edu}

\author{J.~Rudzinski}
\address{Department of Mathematics and Statistics, The University of North Carolina at Greensboro, Greensboro, NC 27402, USA} 
\email{jerudzin@uncg.edu}

\author{C.~Smyth }
\address{Department of Mathematics and Statistics, The University of North Carolina at Greensboro, Greensboro, NC 27402, USA} 
\email{cdsmyth@uncg.edu}
\thanks{Clifford Smyth was supported by NSA MSP Grant H98230-13-1-0222 and by a grant from the Simons Foundation (Grant Number 360486, CS)}
\begin{abstract}
We introduce weighted versions of the classical \v{C}ech and Vietoris-Rips complexes. We show that a version of the Vietoris-Rips Lemma holds for these weighted complexes and that they enjoy appropriate stability properties. We also give some preliminary applications of these weighted complexes.
\end{abstract}

\maketitle
\section{Introduction}

Topological data analysis (TDA) provides a means for the power of algebraic topology to be used to better understand the shape of a dataset. In the traditional approach to TDA, isometric balls of a fixed radius $r>0$ are centered at each data point in some ambient Euclidean space. One then constructs the nerve of the union of these balls and computes the simplicial homology of this nerve. Computationally, this approach is infeasible for large data sets or high-dimensional data, so instead one computes the so-called Vietoris-Rips complex, which is the flag complex over the graph obtained by placing an edge between any pair of vertices that are at distance no more than $2r$ from each other. The key idea of TDA is to allow the radius of these balls to vary and to compute simplicial homology for each value of this radius to create a topological profile of the space. This profile is encoded in either a barcode or a persistence diagram. Topological features such as holes or voids that exist for a relatively large interval of radii are said to persist and are believed to be more important than more transient features that exist for very short intervals of radii. (There are, however, important exceptions to this rule of thumb, see~\cite{Brains}).

In the traditional model, the radius of each ball is the same and can be modeled by the linear function of time $r(t) = r t$.  In this paper, we consider a model of computing persistent homology in which the radius of each ball is allowed to be a different monotonic function $r_x(t)$ at each point $x$.  In this way we can emphasize certain data points by assigning or {\em weighting} them with larger and/or more quickly growing balls and de-emphasize others by weighting them with smaller and/or more slowly growing balls. This is appropriate in the case of a noisy dataset, for instance, as an alternative to throwing away data that fails to meet some threshold of significance. Various other methods of enhancing persistence with weights have been considered (e.g.~\cite{Buchet,Edelsbrunner,Petri,Ren-Further,Ren}).

The weighted model we propose fits into the framework of generalized persistence in the sense of \cite{BdSS}. We show that it enjoys many of the properties familiar from the techniques of traditional persistent homology. We prove a weighted Vietoris-Rips Lemma (Theorem~\ref{thm:multiscale-rips}) that relates our weighted \v{C}ech and Rips complexes in the same way that they are related in the case of isometric balls. We also show that the persistent homology computed over weighted complexes is stable with respect to small perturbations of the rates of growth and/or the points in the dataset (Theorem~\ref{thm:stability}). Moreover, packages for computing persistent homology such as Javaplex~\cite{Javaplex} and Perseus~\cite{Perseus} are capable of handling our weighted persistence with the same complexity as unweighted persistence by merely adjusting inputs to the package functions.

As a proof of concept, we apply our methods to the Modified National Institute of Standards and Technology (MNIST) data set of handwritten digits translated into pixel information. Our method proves more effective than isometric persistence in finding the number $8$ from among these handwritten digits. (We chose $8$ for its unique $1$-dimensional homology among these digits.) We found our methods to be $95.8\%$ accurate as opposed to isometric persistence's $92.07\%$ accuracy. This experiment was chosen to demonstrate the performance of weighted persistence over usual persistence, but it should be noted that neither method approaches the accuracy of state-of-the art computer vision and we make no claim that we are improving on known methods.

In Section 2, we provide the background definitions that are needed for what follows and describe our weighted persistence model. In Section 3 we prove the weighted Vietroris-Rips Lemma and indicate how persistent homology packages can be used to compute weighted persistence. In Section 4 we establish our stability results. Our experiments on MNIST data appear in Section 5. We end with some remarks and questions for further study.

\section{Preliminaries}

We begin by defining some terminology and setting our notation. We will assume some familiarity with simplicial homology and the basic ideas of topological data analysis. For details, we refer to~\cite{Ede:10,Rotman}.

In algebraic topology, simplicial homology is a tool that assigns to any simplicial complex $K$ a collection of $\mathbb{Z}$-modules $H_0(K), H_1(K),\ldots$, called \df{homology groups}, in such a way that the rank of $H_n(K)$ describes the number of ``n-dimensional holes'' in $K$. For our purposes, we replace the standard definition in terms of $\mathbb{Z}$-modules with vector spaces (usually over the field with two elements, for ease of computation). We therefore refer to \df{homology vector spaces} instead of homology groups. We do not attempt to define $H_n(K)$ here, but instead refer to any text in algebraic topology, such as~\cite{Rotman}.

Let $\mathcal{U}$ be a collection of sets. We define the \df{nerve} $\mathcal{N}(\mathcal{U})$ to be the abstract simplicial complex with vertex set $\mathcal{U}$ with the property that the subset $\{U_0,U_1,\ldots, U_n\}$ of $\mathcal{U}$ spans an $n$-simplex in $\mathcal{N}$ whenever $\bigcap_{i=0}^nU_i\neq\emptyset$. 

Let $(X,d)$ be a metric space.  We define $B_r(x) = \{y \in X | d(x,y) <r\}$ and 
$\Ball_r(x)=  \{y \in X | d(x,y) \leq r\}$ to be the open and closed balls of radius $r$ about $x$, respectively. (Note that we're abusing notation since in a general metric space $\Ball_r(x)$ is not necessarily the closure of the open ball, usually denoted $B_r(x)$).  We most often consider examples where $X$ is a subset of $\mathbb{R}^d$ and $d(x,y) = \|x-y\|$ is the Euclidean distance between $x$ and $y$. For a real number $r\ge 0$, we define the \df{\v{C}ech complex of $X$ at scale $r$} by $\Cech(r)=\mathcal{N}\{\Ball_r(x)\mid x\in X\}$.  

We generalize this construction by allowing the radius of the ball around each element $x$ to depend on $x$. Let $\mathbf{r}:X\to[0,\infty)$ be any function. We define the \df{weighted $\mathbf{r}$-$\Cech$ complex} $\Cech(\mathbf{r})$ of $X$ by $\Cech(\mathbf{r})=\mathcal{N}\{\Ball_{\mathbf{r}(x)}(x)\}.$ 

In practice, it is difficult to determine whether an intersection of balls is nonempty. A much simpler construction to use is the Vietoris-Rips complex. For a given parameter $r\ge 0$ the \df{Vietoris-Rips complex} is the flag complex of the $1$-skeleton of the $\Cech$ complex, i.e. a collection of $n+1$ balls forms an $n$-simplex in the Vietoris-Rips complex if and only if the balls are pairwise intersecting.  For the Vietoris-Rips complex we identify each ball with its center, so that the \df{Vietoris-Rips complex at scale $r$} is $\VR(r) =\{\sigma\subset X\mid \diam(\sigma) \leq 2r\}.$ Similarly, if $\mathbf{r}:X\to [0,\infty)$, the \df{weighted $\mathbf{r}$-Vietoris-Rips complex} is $\VR(\mathbf{r}) =\{\sigma\subset X\mid d(x,y) \leq \mathbf{r}(x)+\mathbf{r}(y), \text{for all $x, y \in \sigma$ with $x \neq y$}\}$. 

Fix $\mathbf{r}:X\to[0,\infty)$ and consider the simplicial complex $\Cech(\mathbf{r})$ (or $\VR(\mathbf{r})$). Using simplicial homology with field coefficients, one can associate homology vector spaces $H_*\left(\Cech(\mathbf{r})\right)$ to these simplicial complexes. Whenever $t_0 \le t_1$ there is a natural inclusion map of simplicial complexes given by $\iota:\Cech(t_0\mathbf{r})\to \Cech(t_1\mathbf{r})$ (or the corresponding inclusion of the Vietoris-Rips complexes). By functoriality, there is an induced linear map on homology $\iota_*:H_*\Cech(t_0\mathbf{r}))\to H_*\Cech(t_1\mathbf{r})$.

Let $X\subset\mathbb{R}^d$ be finite. Although we defined the weighted complexes above for any function $\mathbf{r}:X\to[0,\infty)$, we want to study the persistence properties of these weighted complexes. For example, in the case of the weighted \v{C}ech complex, we want to study the evolution of homology as the radii of the balls grow to infinity. One straightforward way to do this would be to simply scale our weighted complexes linearly in the same way that one usually scales the isometric balls in persistent homology. We prefer a more flexible approach, which we describe in terms of radius functions. 

Let $\mathcal{C}^1_+=\mathcal{C}^1_+([0,\infty)]$ denote the collection of differentiable bijective functions $\phi:[0,\infty)\to[0,\infty)$ with positive first derivative. By a \df{radius function} on $X$ we mean a function $\mathbf{r}:X\to \mathcal{C}^1_+$. We denote the image function $\mathbf{r}(x)$ by $\mathbf{r}_x$.

For $t \geq 0$, we define the \df{$\Cech$ and Vietoris-Rips complexes at scale $t$} by \[\Cech_{\mathbf{r}}(t) = \mathcal{N}\{\Ball_{\mathbf{r}_x(t)}(x)\}\] and \[VR_{\mathbf{r}}(t) = \{\sigma \subset X | d(x,y) \leq \mathbf{r}_x(t) + \mathbf{r}_y(t) \text{ for all $x,y \in \sigma$ with $x\neq y$}\},\] respectively.
 We define the \df{entry function}, \begin{equation} \label{eq:entry_function}
      f_{X,\mathbf{r}}(y)=\min_{x\in X}\{\mathbf{r}^{-1}_{x}(d(y,x))\}\hbox{.}\end{equation} 
      
      This function captures the scale $t$ at which the point $y\in\mathbb{R}^d$ is first captured by some ball $\Ball_{\mathbf{r}_x(t)}(x)$; we have $f_{X, \mathbf{r}}(y) = t$ if and only if $y \in \Ball_{\mathbf{r}_x(t)}(x)$ for some $x$ in $X$ and $y \not \in \bigcup_{x \in X} B_{\mathbf{r}_x(t)}(x)$. Thus we have the following proposition.

\begin{proposition}
Let $X$ be a finite subset of some Euclidean space $\mathbb{R}^d$. Suppose that $\mathbf{r}$ and $f_{X,\mathbf{r}}$ are defined as above. Then, \[f^{-1}_{X,\mathbf{r}}\left( [0,t]\right)=\bigcup_{x \in X}\Ball\left(x,\mathbf{r}_{x}(t)\right)\hbox{.}\]
\end{proposition}

It follows from the Nerve Lemma (see for example, \cite[Corollary 4G.3]{Hatcher}), that $\Cech_{\mathbf{r}}(t)$ is homotopy equivalent to $f^{-1}_{X,\mathbf{r}}\left([0,t]\right)$.

\section{A weighted Vietoris-Rips lemma}

The Vietoris-Rips complex is much easier to compute than the \v{C}ech complex in high dimensions. To determine whether $n+1$ balls form an $n$-simplex in the \v{C}ech complex, we must check whether the balls intersect, a computationally complex problem. To determine whether $n+1$ balls $B_{r_i}(x_i)$ form a simplex in the Vietoris-Rips complex is computationally easy, only $\binom{n+1}{2}$ conditions $d(x_i,x_j) \leq r_i +r_j$ need be checked.  Furthermore, if there are $m$ points in $X$, it may be necessary to check all $2^m$ sub-collections of balls to determine the \v{C}ech complex, whereas determining the Rips complex will only require checking $\binom{m}{2}$ pairs of points.

Our weighted \v{C}ech and Vietoris-Rips complexes are similar in spirit to weighted alpha complexes~\cite[III.4]{Ede:10}. Both constructions seek to permit ``balls'' with different sizes. Our constructions are simpler from a conceptual standpoint since the alpha complexes are built as subcomplexes of the Delaunay complex, which comes from the Voronoi diagram. Moreover, our complexes are computationally simple; indeed our method of finding weighted Vietoris-Rips complexes requires only marginally more computation than the unweighted Vietoris-Rips complex.

In particular, Javaplex and Perseus can compute regular (unweighted) persistent homology given input of a distance matrix $M$ with $M_{i,j}=d(x_i,x_j)$. Inputting $M_{i,j}=d(x_i,x_j)/(r_i+r_j)$ allows these packages to compute the persistent homology with $r_{x_i}(t)=r_it$ in the same time.

In computational problems it is common to use the Vietoris-Rips complex instead of the $\Cech$ complex to simplify the calculational overhead. The following theorem justifies this decision by saying that the Vietoris-Rips complex is ``close'' to the weighted \v{C}ech complex. 

The classical Vietoris-Rips Lemma can be stated as follows:
\begin{theorem}\cite{dSG-2007}
Let $X$ be a set of points in $\mathbb{R}^d$ and let $t > 0$.  Then \[\VR(t') \subseteq \Cech(t) \subseteq \VR(t)\] whenever $0< t' \le t \left(\sqrt{2d/(d+1)}\right)^{-1}$.
\end{theorem}

The main result of this section is an extension of this result to the weighted case.  

\begin{theorem}[Weighted Vietoris-Rips Lemma] \label{thm:multiscale-rips} Let $X$ be a set of points in $\RR^d$.  Let $\mathbf{r}: X \to (0,\infty)$ be the corresponding weight function and let $t >0$. Then \[\VR(t'\mathbf{r}) \subseteq \Cech(t \mathbf{r}) \subseteq \VR(t \mathbf{r})\] whenever $0< t' \le t \left(\sqrt{2d/(d+1)}\right)^{-1}$.
\end{theorem}

\begin{proof}
The second containment $\Cech(t \mathbf{r}) \subseteq \VR(t \mathbf{r})$ follows from the fact that the weighted Vietoris-Rips complex is the flag complex of the weighted \v{C}ech complex. 
        
To show that $\VR(t'\mathbf{r})\subset \Cech(t\mathbf{r})$, we suppose there is some finite collection $\sigma = \{x_k\}_{k=0}^\ell \subseteq \RR^d$ with $\ell >0$ that is a simplex in $\VR(t' \mathbf{r})$ and show that this is also a simplex in $\Cech(t \mathbf{r})$. We have $\|x_i-x_j\|_2 \le t' (\mathbf{r}(x_i) + \mathbf{r}(x_j))$ whenever $i \neq j$.

Define a function $f:\RR^d \to \RR$ by \[f(y) = \max_{0 \leq j \leq \ell} \left\{\frac{\|x_j - y\|_2}{\mathbf{r}(x_j)}\right\}.\] Clearly, $f$ is continuous and $f(y) \to \infty$ as $\|y\|_2 \to \infty$. Thus $f$ attains a minimum (say at $y_0$) on some compact set containing $\Conv(\{x_k\}_{k=0}^{\ell})$. (Here $\Conv(S)$ is the convex hull of the set $S \subseteq {\mathbb R}^d$.) We must have $\|x_i-y_0\|_2/ r(x_i) = f(y_0)$ for at least one of the vertices $x_i$. By reordering the vertices, we may assume that  \[f(y_0) = \frac{1}{\mathbf{r}(x_j)} \|x_j - y_0\|_2\qquad \hbox{if $0\le j\le n$}\] and \[f(y_0) > \frac{1}{\mathbf{r}(x_j)} \|x_j - y_0\|_2\qquad \hbox{if $n<j\le \ell$.}\] Let \[g(y) = \max_{0 \leq j \leq n} \left\{\frac{1}{\mathbf{r}(x_j)}\left\|x_j - y\right\|_2\right\}\] and \[h(y) = \max_{n < j \leq \ell} \left\{\frac{1}{\mathbf{r}(x_j)}\left\|x_j - y\right\|_2\right\}.\]
		
Now we wish to show that $y_0 \in \mathrm{Conv}(\{x_j\}_{j=0}^{n})$. To this end we apply the Separation Theorem \cite{Matousek} to obtain: either $y_0 \in \Conv(\{x_j\}_{j=0}^{n})$ or there is a $v \in \RR^d$ and a $C<0$ such that $v \cdot x_j \ge 0$ for all $0 \le j \le n$ and $v \cdot y_0 < C$. Thus if $y_0 \not \in \Conv(\{x_j\}_{j=0}^{n})$ there is a $v \in \RR^d$ so that $v \cdot (x_j - y_0) > 0$ for $0 \leq j \leq n$. We suppose that there is such a $v$ and derive a contraction.

Since \[\left\|x_j - (y_0 + \lambda v)\right\|_2^2 = \left\|x_j - y_0\right\|_2^2 - 2 \lambda v \cdot (x_j - y_0) + \lambda^2 \left\|v\right\|_2^2\] for each $0 \leq j \leq n$, it follows that $g(y_0 + \lambda v) < f(y_0)$ for all $0 < \lambda < \lambda_1$, where \[\lambda_1 = \min_{0 \leq j \leq n} \left\{\frac{2 v \cdot (x_j - y_0)}{\left\|v\right\|_2^2}\right\}.\] Since $h(y)$ is continuous and $h(y_0) < f(y_0)$, there exists a $\lambda_2$ so that $h(y_0 + \lambda v) < f(y_0)$ for $0 < \lambda < \lambda_2$. Thus, there exists a $\lambda > 0$ such that \[f(y_0 + \lambda v) = \max\left\{g(y_0 + \lambda v), h(y_0 + \lambda v)\right\} < f(y_0)\hbox{,}\] contradicting the minimality of $y_0$.

By Carath\'{e}odory's theorem \cite{Matousek} and reordering of vertices if necessary, $y_0$ is a convex combination of some subcollection of vertices $\{x_j\}_{j=0}^{m}$ where $m \leq \min\{d,n\}$. It is not possible that $m = 0$. If so, then $y_0 = x_0$ and $f(y_0) = \frac{1}{\mathbf{r}(x_0)}\|x_0 - y_0\|_2 = 0$ and f is identically zero. Since $\sigma$ has dimension at least $1$, it contains a vertex $x_1 \neq x_0$. It follows that $f(y_0) = f(x_0) > \frac{1}{\mathbf{r}(x_1)} \|x_1 - x_0\|_2 > 0$, which is a contradiction. 
		
Let $\widehat{x_j} = x_j - y_0$ for all $0 \leq j \leq m$. Note that 
\begin{equation}
\|\widehat{x_j}\|_2^2 = \mathbf{r}(x_j)^2 f(y_0)^2.
\end{equation} 
Since $y_0 \in \Conv(\{x_j\}_{j=0}^m)$, $y_0 = \sum_{j=0}^{m} a_j x_j$ for some set of non-negative real numbers $a_0,\ldots, a_m$ that sum to $1$. Thus $\sum_{j=0}^{m} a_j \widehat{x_j}=0$. By relabeling, we may assume that $a_0 \mathbf{r}(x_0) \ge a_j\mathbf{r}(x_j)$ when $j > 0$. Necessarily $a_0 >0$. (Otherwise $a_j = 0$ for all $0 \leq j \leq m$, a contradiction.) Then, \[\widehat{x_0} = -\sum\limits_{j=0}^{m} \frac{a_j}{a_0} \widehat{x}_j\] and so \[\mathbf{r}(x_0)^2 f(y_0)^2 = \|\widehat{x_0}\|_2^2 = - \sum\limits_{j=0}^{m} \frac{a_i}{a_0} \widehat{x}_0 \cdot \widehat{x_j}.\]
		
Among the indices $1,2,\ldots,m$, there is some $j_0$ such that
\begin{equation}
\frac{1}{d} \mathbf{r}(x_0)^2 f(y_0)^2 \leq \frac{1}{m} \mathbf{r}(x_0)^2 f(y_0)^2 \leq -\frac{a_{j_0}}{a_0} \widehat{x_0} \cdot \widehat{x_{j_0}}.
\end{equation}
We must have $a_{j_0}>0$.  (Otherwise, $f(y_0) = 0$, which, as shown earlier, is a contradiction.)
By reordering, we may assume $j_0 = 1$.
Putting (1) and (2) together, we find 
\begin{align*}
f(y_0)^2 \left( \mathbf{r}(x_0)^2 + \frac{2 a_0 \mathbf{r}(x_0)^2}{a_1 d} + \mathbf{r}(x_{1})^2\right) & = f(y_0)^2\mathbf{r}(x_0)^2+\frac{2a_0f(y_0)^2\mathbf{r}(x_0)^2}{a_1d}+f(y_0)^2\mathbf{r}(x_1)^2 \\
&\leq \left\|\widehat{x_0}\right\|_2^2 - 2 \widehat{x_0} \cdot \widehat{x_1} + \left\|				\widehat{x_1}\right\|_2^2 \\
			&= \left\|\widehat{x_0} - \widehat{x_1}\right\|_2^2 \\
			&= \left\|x_0 - x_1\right\|_2^2 \\
			&\leq \left(t' \left(\mathbf{r}(x_0)+\mathbf{r}(x_1)\right)\right)^2.
\end{align*}

We will now show that \[\frac{f(y_0)^2}{t'} \leq \frac{(\mathbf{r}(x_0)^2 + \mathbf{r}(x_1)^2)^2}{\mathbf{r}(x_0)^2 + \frac{2 a_0 \mathbf{r}(x_0)^2}{a_1 d} + \mathbf{r}(x_1)^2} \leq \frac{2d}{d+1}.\] It suffices to show, after cross-multiplying the right-hand inequality, that \[(d - 1 + 4\frac{a_0}{a_1}) \mathbf{r}(x_0)^2 - 2(d+1) \mathbf{r}(x_0) \mathbf{r}(x_1) + (d-1) \mathbf{r}(x_1)^2 \geq 0.\] Since $\dfrac{a_0}{a_1} \geq \dfrac{\mathbf{r}(x_1)}{\mathbf{r}(x_0)}$ we get
\begin{align*}
\left(d - 1 + 4\frac{a_0}{a_1}\right) \mathbf{r}(x_0)^2 &- 2(d+1) \mathbf{r}(x_0) \mathbf{r}(x_1) + (d-1) \mathbf{r}(x_1)^2 \\
			&\geq \left(d - 1 + 4\frac{\mathbf{r}(x_1)}{\mathbf{r}(x_0)}\right) \mathbf{r}(x_0)^2 - 2(d+1) \mathbf{r}(x_0) \mathbf{r}(x_1) + (d-1) \mathbf{r}(x_1)^2 \\
			&= (d-1)\left(\mathbf{r}(x_0) - \mathbf{r}(x_1)\right)^2 \\
			&\geq 0
\end{align*}
as desired. Our assumption that $t' \le t (\sqrt{2d/(d+1)})^{-1}$ implies $f(y_0) \leq t$ and thus \[y_0 \in \bigcap\limits_{i = 0}^{\ell} \Ball_{t \mathbf{r}(x_i)} (x_i).\] Therefore $\sigma \in \check{C}(t\mathbf{r})$ and we are done.
	\end{proof}

\section{Stability} 

In this section we discuss the stability of our weighted persistence.  Let $X$ and $Y$ be finite subsets of $\mathbb{R}^d$ with corresponding radii functionals ${\mathbf r}: X \to \mathcal{C}^1_+$ and $\mathbf{s}: Y \to \mathcal{C}^1_+$.  Informally, we show that if $(X, \mathbf{r})$ and $(Y, \mathbf{s})$ are ``close'', i.e. are small perturbations of each other, then the corresponding entry functions $f_{X, \mathbf{r}}$ and $f_{Y,\mathbf{s}}$ (see \eqref{eq:entry_function}) are also ``close'' and hence the associated persistence diagrams must also be ``close''.  We'll now make the the definitions of these various types of closeness precise.

Let $\eta \subseteq X \times Y$ be a relation such that for every $x \in X$ there is a $y \in Y$ with $(x,y) \in \eta$ and for every $y \in Y$ there is an $x \in X$ with $(x,y) \in \eta$.  We measure the closeness of $X$ and $Y$ with respect to $\eta$ by \[\|\eta\| := \max_{(x,y) \in \eta} d(x,y).\]  If $L$ is any compact set and $h : L \to \mathbb{R}$ is continuous let \[\|h\|_L := \max_{x \in L} |h(x)|.\] Let $K$ be a compact subset of $\mathbb{R}^d$ that contains $X \cup Y$.  The closeness of $\mathbf{r}$ and $\mathbf{s}$ is measured by \[D(\mathbf{r}, \mathbf{s})_{\eta,K} := \max_{(x,y) \in \eta} \|\mathbf{r}_x^{-1} - \mathbf{s}_y^{-1}\|_{[0,\diam(K)]}.\]  
The closeness of $f_{X, \mathbf{r}}$ and $f_{Y,\mathbf{s}}$ is measured by $\|f_{X, \mathbf{r}} - f_{Y,\mathbf{s}}\|_K.$  We also define $S(\mathbf{r})_K := \max_{x \in X} \|(\mathbf{r}^{-1}_x)'\|_{[0,\diam(K)]}$.

As is common, we measure the closeness of persistence diagrams by the bottle-neck distance.  We'll give the definition of this metric in the remarks leading up to Theorem 4.5.

\begin{theorem} \label{thm:stability}
In the above notation we have the following bound on entry functions (see \eqref{eq:entry_function}): 

\[\|f_{X,\mathbf{r}}-f_{Y,\mathbf{s}}\|_K\le D(\mathbf{r},\mathbf{s})_{\eta,K}+ \|\eta\| \max( S(\mathbf{r})_K, S(\mathbf{s})_K)\]
\end{theorem}

\begin{proof}


There is some point $z$ in the compact set $K$ and some points $x \in X$ and $y \in Y$ so that
\[\|f_{X,\mathbf{r}}-f_{Y,\mathbf{s}}\|_{K}=|f_{X,\mathbf{r}}(z)-f_{Y,\mathbf{s}}(z)| = |\mathbf{r}^{-1}_{x}(d(z,x))-\mathbf{s}^{-1}_y(d(z,y))|.\]

We first suppose $\mathbf{r}^{-1}_{x}(d(z,x)) \geq \mathbf{s}^{-1}_{y}(d(z,y))$. Let $x' \in X$ such that $(x',y) \in \eta$.  Since $f_{X,\mathbf{r}}$ is a minimum, $\mathbf{r}^{-1}_{x'}(d(z,x')) \geq \mathbf{r}^{-1}_{x}(d(z,x))$ and we have
\begin{equation} \label{eq:ineq} \|f_{X,\mathbf{r}}-f_{Y,\mathbf{s}}\|_{K} \leq |\mathbf{r}^{-1}_{x'}(d(z,x'))-\mathbf{s}^{-1}_{y}(d(z,y))| \end{equation}
\[\leq |\mathbf{r}^{-1}_{x'}(d(z,x'))-\mathbf{s}^{-1}_{y}(d(z,x'))| + |\mathbf{s}^{-1}_{y}(d(z,x'))-\mathbf{s}^{-1}_{y}(d(z,y))|\]
Since $d(z,x') \in [0,\diam(K)]$, 
\[|\mathbf{r}^{-1}_{x'}(d(z,x'))-\mathbf{s}^{-1}_{y}(d(z,x'))| \leq D(\mathbf{r},\mathbf{s})_{\eta,K}.\]
Since $|d(z,x') - d(z,y)| \leq d(x',y) \leq \|\eta\|$ we apply the mean value theorem to obtain the bound
\[|\mathbf{s}^{-1}_{y}(d(z,x'))-\mathbf{s}^{-1}_{y}(d(z,y))|\le \|\eta\| \cdot \|(\mathbf{s}^{-1}_y)'\|_{[0,\diam(K)]} \leq \|\eta\| \max( S(\mathbf{r})_K, S(\mathbf{s})_K).\]
Together, these last two bounds give the bound of the theorem.  A similar argument gives the same bound if $\mathbf{r}^{-1}_{x}(d(z,x)) \leq \mathbf{s}^{-1}_{y}(d(z,y))$.
\end{proof}

If one has free choice of the perturbed set $(Y,\mathbf{s})$ it is clear that $\|f_{X,\mathbf{r}} - f_{Y,\mathbf{s}}\|_K$ can be made arbitrarily large. This could be done, say by adding a point to $Y$ that is arbitrarily far from any point in $X$ or by making one $\mathbf{s}_y$ arbitrarily larger than any $\mathbf{r}_x$.  The upper bound of Theorem 4.1 is also a bound on how extreme such perturbations may be.

We have the following immediate corollary of Theorem~\ref{thm:stability}.

\begin{corollary}
If the radii functions are all linear, i.e. if there are positive constants $r_{x}$ and $s_{y}$ for all $x \in X$ and $y \in Y$ such that $r_{x}(t) = r_{x} t$ and $s_{y}(t) = s_{y} t$, then \[\|f_{X,r}-f_{Y,s}\|_K\le \diam(K) \max_{(x,y) \in \eta} \left|\frac{1}{r_{x}}- \frac{1}{s_{y}}\right|+ \|\eta\| \max\left( \max_{x \in X} \frac{1}{r_{x}} , \max_{y \in Y} \frac{1}{s_{y}}\right)\hbox{.}\]

\end{corollary}

For our next two corollaries, let $X$ and $Y$ have the same cardinality and let $m : X \to Y$ be a bijection.  We now consider each point $x \in X$ as being perturbed to a point $m(x) \in Y$ and hence set $\eta = \{ (x,m(x)) : x \in X\}$. We have the following point stability result in which the points are perturbed but the weight functions stay the same.  

\begin{corollary}(Point-stability)
If only the locations of the points are perturbed and the radius functions stay the same, i.e. $\mathbf{s}_{m(x)}(t) = \mathbf{r}_{x}(t)$ for all $x \in X$, then
\[\|f_{X,\mathbf{r}}-f_{Y,\mathbf{s}}\|_K\le \max_{x \in X} d(x,m(x)) \left\|(\mathbf{r}^{-1}_{x})'\right\|_{[0,\diam(k)]}\hbox{.}\]
\end{corollary}

\begin{proof}
We follow the proof of Theorem~\ref{thm:stability}.  Take $x' \in X$ such that $m(x')=y$.  Then $\mathbf{S}_y = \mathbf{r}_{x'}$ and the first term in the upper bound in inequality \eqref{eq:ineq} is $0$.  Since the second term in that upper bound is bounded above by $d(x',m(x')) \|(\mathbf{r}_{x'}^{-1})'\|_{[0,\diam(K)]}$, the bound of the corollary holds.
\end{proof}

The next corollary is a weight-function stability result concerning a case in which the points stay the same ($Y = X$ and $m(x)=x$) but the weight functions are perturbed.
\begin{corollary}(Weight-function stability)
If only the radii functions are perturbed and the points stay the same then, 
\[\|f_{X,\mathbf{r}}-f_{X,\mathbf{s}}\|_K\le \max_{x \in X} \|\mathbf{r}_{x}^{-1}-\mathbf{s}_{x}^{-1}\|_{[0,\diam(K)]}\hbox{.}\]
\end{corollary}

\begin{proof}
Again following the proof of Theorem~\ref{thm:stability} we take $x'=m(x')=y$.  Now the second term in the upper bound of \eqref{eq:ineq} is $0$ and the first term is $|\mathbf{r}^{-1}_{y}-\mathbf{s}^{-1}_{y}|$ where $t=d(z,y)$. The corollary follows.
\end{proof}

We now show the stability of the persistence diagrams of $f_{X,\mathbf{r}}$ under perturbations of $X$ and $\mathbf{r}$.  Let $f: K \to [0,\infty)$ be a real valued function on a compact set $K \subseteq \mathbb{R}^d$. The \df{persistence diagram} of $f$, $\dgm(f)$, is a multi-set of points in $[0, +\infty]^2$ recording the appearance and disappearance of homological features in $f^{-1}([0,t])$ as $t$ increases.  Each point $(b,d)$ in the diagram tracks a single homological feature, recording the scale $t=b$ at which the feature first appears and the scale $t=d$ at which it disappears \cite{Ede:10}. It should also be noted that if one considers the birth-death pair as an interval we obtain the \df{barcode} as seen in \cite{Zom:05} (see Figures 2 and 3).  Given two functions $f,g : K \to [0, \infty]$ let $P = \dgm(f)$ and $Q = \dgm(g)$ be the corresponding persistence diagrams (where as usual we include all points along the diagonal in $P$ and $Q$). We let $N$ denote the set of all bijections from $P$ to $Q$. We recall that the \df{bottleneck distance} between the diagrams~\cite{Ede:10} is given by
\[
d_B(\dgm(f),\dgm(g)) = \inf_{\gamma \in N} \sup_{x \in P} ||x - \gamma(x)||_\infty\text{.}
\]
We have 
\begin{theorem}\cite[Theorem 6.9]{CSEH} Suppose $\mathcal{X}$ is a triangulable space and that $f:\mathcal{X}\to \mathbb{R}$ and $g:\mathcal{X}\to \mathbb{R}$ are tame, continuous functions. If $|f-g|$ is bounded, then for each $n$,
\[d_B(\dgm_n(f), \dgm_n(g))\le \|f-g\|_\infty\]
where $d_B$ denotes the bottleneck distance and $\dgm_n(f)$ denotes the $n$-th persistence diagram of the filtration of $f$.
\end{theorem}
We refer to \cite{Ede:10} for the technical definitions of tame and triangulable.  Note that as our spaces are nerves of balls around finite collections of points, they are finite simplicial complexes.  Hence they are triangulable and only admit tame functions.  Thus for our setting we get the following corollary.
\begin{corollary}
Let $X$ and $Y$ be finite subsets of $\mathbb{R}^d$ and let  $\mathbf{r}:X \to {\mathcal C}^1_+$ and $\mathbf{s} : Y \to {\mathcal C}^1_+$.  Suppose that $\eta\subseteq X\times Y$ is a relation as above and $K$ is a compact subset of $\mathbb{R}^d$ containing $X$ and $Y$.  Then for each $n$,
\[d_B(\dgm_n(f_{X,\mathbf{r}}),\dgm_n(f_{Y,\mathbf{s}})) \le D(\mathbf{r},\mathbf{s})_{\eta,K}+ \|\eta\| \max( S(\mathbf{r})_K, S(\mathbf{s})_K)\hbox{.}\] 
\end{corollary}

\section{MNIST eights recognition}

In this section, we give an application of weighted persistence to a simple computer vision problem. We apply our methods to the Modified National Institute of Standards and Technology (MNIST) data set of handwritten digits. We should emphasize that this application is simply a proof of concept; our methods to detect the handwritten number 8 fall well short of state-of-the-art methods~\cite{Cir:12}.

The MNIST dataset consists of handwritten digits (0 through 9) translated into pixel information. Each data point contains a label and 784 other values ranging from 0 to 255 that correspond to a 28 by 28 grid of pixels. The values 0 through 255 correspond to the intensity of the pixels in gray-scale with 0 meaning completely black and 255 meaning completely white. Considering the digits from zero through nine, unweighted persistence would easily be able to classify these numbers as having zero, one, or two holes, provided they are written precisely; however, real handwritten digits present a challenge.  Consider an eight as in Figure~\ref{Crazy8}. Unweighted persistence would pick up on two holes, but one of those holes might be slightly too small and ultimately considered insignificant, see Figure~\ref{UnweightedBar}. Our methods are able to pick up on both holes and would count them as significant, see Figure~\ref{WeightedBar}. We chose to work with the digit eight due to its unique homology.

\begin{figure}[t]
\begin{center}
\includegraphics[height = 3cm]{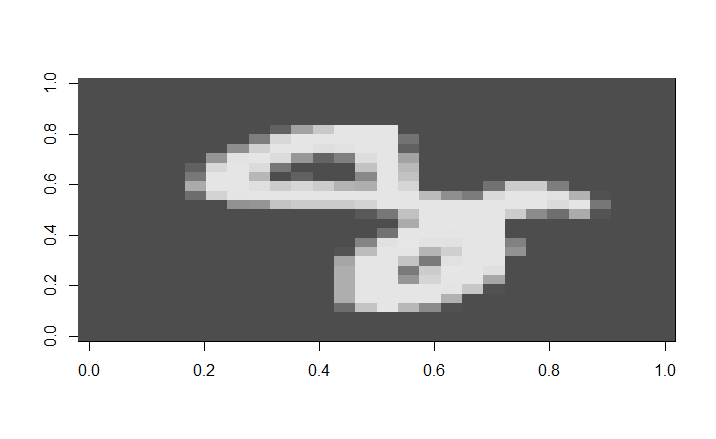}
\caption{An eight converted to a 28 by 28 grid of pixels}\label{Crazy8}
\end{center}
\end{figure}

To begin, we convert each 28 by 28 to a set of points in the plane. We treat the location of a value in the matrix as a location in the plane. That is, the value in the $i$th row, $j$th column corresponds to the point $(i,j)$. The weight on each point is exactly its corresponding pixel intensity. Using this set of points and corresponding weights we calculate persistent homology via weighted Rips complexes. We test this methods performance against the unweighted case where all nonzero pixel values have the uniform weight of 1, again we calculate persistence in this case via Rips complexes.

We compare weighted persistence to unweighted persistence by measuring accuracy of classifying eights. Notice in the barcodes that the deciding factor in determining an eight is the ability to distinguish the length of the second longest bar from the length of the third longest and smaller bars. For this reason, we consider the ratio of the third longest bar to the second longest bar. We will say (arbitrarily) that a barcode represents an eight if this ratio is less than $\frac{1}{2}$. For each of the 42,000 handwritten digits in the MNIST data set, we compute both weighted and unweighted persistence and collect the predictions. We obtain the confusion matrices as in Figure~\ref{conf-mat}.

\begin{figure}
\begin{center}
\includegraphics[height = 5cm]{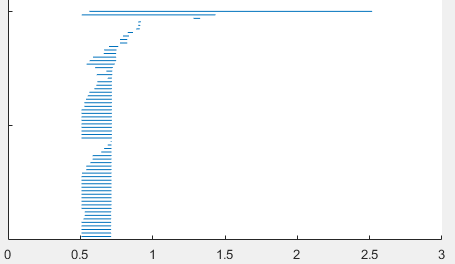}
\caption{Weighted persistence on the image from Figure 1 produces a barcode that clearly has two long bars in dimension 1.}\label{WeightedBar}
\end{center}
\end{figure}

\begin{figure}
\begin{center}
\includegraphics[height = 2.5cm]{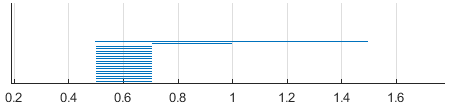}
\caption{Unweighted persistence on the image from Figure 1 produces a barcode that has one long bar (in $1$-homology). The second longest bar is hard to distinguish (in length) from the rest.}\label{UnweightedBar}
\end{center}
\end{figure}

\begin{figure}
\begin{center}
    \begin{tabular}{|c|c|c||c|c|}\hline
    & \multicolumn{2}{c||}{Weighted Persistence} & \multicolumn{2}{c|}{Unweighted Persistence}\\ \hline
    & Predicted not $8$ & Predicted $8$ & Predicted not $8$ & Predicted $8$\\ \hline
    Not $8$ & 36487 & 1450 & 35869 & 2068 \\ \hline
    Is $8$ & 633 & 3430 & 1261 & 2802\\ \hline
\end{tabular}\caption{The confusion matrices show that weighted persistence outperforms its unweighted counterpart.}\label{conf-mat}
\end{center}

\end{figure}

Notice that the weighted persistence has an accuracy rate of 95.8\% whereas unweighted persistence had an accuracy of 92.07\%. A full summary can be seen in Figure \ref{summarytable}.  We view this result as promising for potential future applications of weighted persistence. 

\begin{figure}[ht]
\begin{center}
\begin{tabular}{|c|c||c|}\hline
& Weighted Persistence & Unweighted Persistence \\ \hline
Accuracy & 0.9504 & 0.9207\\ \hline
Sensitivity & 0.9618 & 0.9455\\ \hline
Specificity & 0.8442 & 0.6896\\ \hline

Pos. Pred. Value & 0.9829 & 0.966\\ \hline

Neg. Pred. Value & 0.7029 & 0.5754\\ \hline

Prevalence & 0.9033 & 0.9033\\ \hline

Balanced Accuracy & 0.903 & 0.8176 \\ \hline

\end{tabular}
\caption{Weighted and unweighted persistence compared.}\label{summarytable}
\end{center}
\end{figure}
\section{Concluding remarks and open questions} 
The method of weighted persistence satisfies the appropriate Vietoris-Rips Lemma, is stable under small perturbations of the points, or the weights, or both, and can be successfully applied to data such as the MNIST data set to improve upon usual persistence. Furthermore, it is just as easy to calculate weighted persistence for balls growing at linear rates, as it is to calculate regular persistence. We conclude the paper with some further observations and questions.

One can imagine weighted persistence as interpolating between two extreme approaches to a data set that is partitioned into data $D$ and noise $N$. More precisely, we consider a noisy data set $X$. Various methods exist to filter $X$ into data $D$ and noise $N$. Traditional persistence can be applied to $D\cup N$ in two ways. We can either assign the same radius to every point of $D\cup N$ or we can throw the points of $N$ out entirely and compute persistence on $D$ alone. Using weighted persistence, we can assign the radius $0$ to each point of $N$ and compute weighted persistence of $D\cup N$. It is easy to see that this will differ from persistence of $D$ itself only in dimension $0$. By gradually increasing the $N$-radii from $0$ to $1$, our stability results can be interpreted as producing a continuum of barcodes/persistence diagrams that interpolate between the usual persistence applied to $D$ and the usual persistence applied to $D\cup N$ (in dimensions above $0$), see~\cite{Austin}.

As mentioned in the introduction, weighted persistence fits into the framework of generalized persistence in the sense of \cite{BdSS}. This direction was explored in detail in \cite{Josh}.

Finally, it would be interesting to apply weighted persistence to the MNIST data set to determine its effectiveness in distinguishing the $1$-homology of the other nine digits. One complication is that the number $4$ presents an interesting challenge since it is appropriate to write it both as a simply connected space and as a space with non-trivial $H_1$. Distinguishing $1$-homology creates 3 clusters of digits from which we could use other machine learning techniques to create an ensemble and make accurate predictions.
\bibliography{references}

\begin{thebibliography}{10}

\bibitem{Brains}
P.~Bendich, J.~S. Marron, E.~Miller, A.~Pieloch, and S.~Skwerer.
\newblock Persistent homology analysis of brain artery trees.
\newblock {\em Ann. Appl. Stat.}, 10(1):198--218, 03 2016.

\bibitem{BdSS}
P.~Bubenik, V.~de~Silva, and J.~Scott.
\newblock Metrics for generalized persistence modules.
\newblock {\em Found. Comput. Math.}, 15(6):1501--1531, 2015.

\bibitem{Buchet}
M.~Buchet, F.~Chazal, S.~Y. Oudot, and D.~R. Sheehy.
\newblock Efficient and robust persistent homology for measures.
\newblock {\em Comput. Geom.}, 58:70--96, 2016.

\bibitem{Cir:12}
D.~{Cire{\c s}an}, U.~{Meier}, and J.~{Schmidhuber}.
\newblock Multi-column deep neural networks for image classification.
\newblock In {\em 2012 IEEE Conference on Computer Vision and Pattern
  Recognition}, pages 3642--3649. IEEE, 2012.

\bibitem{CSEH}
D.~Cohen-Steiner, H.~Edelsbrunner, and J.~Harer.
\newblock Stability of persistence diagrams.
\newblock {\em Discrete Comput. Geom.}, 37(1):103--120, 2007.

\bibitem{dSG-2007}
V.~de~Silva and R.~Ghrist.
\newblock Coverage in sensor networks via persistent homology.
\newblock {\em Algebr. Geom. Topol.}, 7:339--358, 2007.

\bibitem{Ede:10}
H.~Edelsbrunner and J.~L. Harer.
\newblock {\em Computational topology}.
\newblock American Mathematical Society, Providence, RI, 2010.
\newblock An introduction.

\bibitem{Edelsbrunner}
H.~Edelsbrunner and D.~Morozov.
\newblock Persistent homology: theory and practice.
\newblock In {\em European {C}ongress of {M}athematics}, pages 31--50. Eur.
  Math. Soc., Z\"urich, 2013.

\bibitem{Hatcher}
A.~Hatcher.
\newblock {\em Algebraic topology}.
\newblock Cambridge University Press, Cambridge, 2002.

\bibitem{Austin}
A.~Lawson.
\newblock Multiscale persistent homology.
\newblock Master's thesis, University of North Carolina at Greensboro, 2016.

\bibitem{Josh}
J.~Martin.
\newblock Multiradial (multi)filtrations and persistent homology.
\newblock Master's thesis, University of North Carolina at Greensboro, 2016.

\bibitem{Matousek}
J.~Matou\v{s}ek.
\newblock {\em Lectures on Discrete Geometry}.
\newblock Springer, 2002.

\bibitem{Perseus}
K.~Mischaikow and V.~Nanda.
\newblock Morse theory for filtrations and efficient computation of persistent
  homology.
\newblock {\em Discrete Comput. Geom.}, 50(2):330--353, 2013.

\bibitem{Petri}
G.~{Petri}, M.~{Scolamiero}, I.~{Donato}, and F.~{Vaccarino}.
\newblock {Topological Strata of Weighted Complex Networks}.
\newblock {\em PLoS ONE}, 8:e66506, June 2013.

\bibitem{Ren-Further}
S.~{Ren}, C.~{Wu}, and J.~{Wu}.
\newblock {Further Properties and Applications of Weighted Persistent
  Homology}.
\newblock {\em ArXiv e-prints}, Nov. 2017.

\bibitem{Ren}
S.~{Ren}, C.~{Wu}, and J.~{Wu}.
\newblock {Weighted Persistent Homology}.
\newblock {\em ArXiv e-prints}, Aug. 2017.

\bibitem{Rotman}
J.~J. Rotman.
\newblock {\em An introduction to algebraic topology}, volume 119 of {\em
  Graduate Texts in Mathematics}.
\newblock Springer-Verlag, New York, 1988.

\bibitem{Javaplex}
A.~Tausz, M.~Vejdemo-Johansson, and H.~Adams.
\newblock Java{P}lex: {A} research software package for persistent
  (co)homology.
\newblock In H.~Hong and C.~Yap, editors, {\em Proceedings of ICMS 2014},
  Lecture Notes in Computer Science 8592, pages 129--136, 2014.
\newblock Software available at
  \url{http://appliedtopology.github.io/javaplex/}.

\bibitem{Zom:05}
A.~Zomorodian and G.~Carlsson.
\newblock Computing persistent homology.
\newblock {\em Discrete Comput. Geom.}, 33(2):249--274, 2005.

\end{thebibliography}
\end{document}